\documentclass{article}
\usepackage{tikz}
\usepackage{graphicx}
\usepackage{pgfplots}
\usepackage{verbatim}
\usepackage{mathtools}
\usepackage{authblk}
\usepackage{amsmath}
\usepackage{amsthm}
\usepackage{amssymb}
\usepackage{caption}
\usepackage{subcaption}
\usepackage{hyperref}

\newtheorem{theorem}{Theorem}
\newtheorem{lemma}{Lemma}
\newtheorem{corollary}{Corollary}
\newtheorem{proposition}{Proposition}

\theoremstyle{definition}

\pgfplotsset{compat=1.18}

\newcommand*{\AAA}{\mathcal{A}}
\newcommand*{\BB}{\mathcal{B}}
\newcommand*{\MM}{\mathfrak{M}}
\newcommand*{\game}{\mathrm{FS}^\tau}
\newcommand*{\Var}{\mathit{Var}}
\newcommand*{\FO}{\mathrm{FO}}

\title{Description Complexity of Unary Structures in First-Order Logic with Links to Entropy} %TODO Please add

\author{Reijo Jaakkola}
\author{Antti Kuusisto}
\author{Miikka Vilander}
\affil{Tampere University}

\date{}

\begin{document}

\maketitle

\begin{abstract}
The description complexity of a model is the length of the shortest formula that defines the model. We study the description complexity of unary structures in first-order logic FO, also drawing links to semantic complexity in the form of entropy. The class of unary structures provides, e.g., a simple way to represent tabular Boolean data sets as relational structures. We define structures with FO-formulas that are  strictly linear in the size of the model as opposed to using the naive quadratic ones, and we use arguments based on formula size games to obtain related lower bounds for description complexity. For a typical structure the upper and lower bounds in fact match up to a sublinear term, leading to a precise asymptotic result on the expected description complexity of a randomly selected structure. We then give bounds on the relationship between Shannon entropy and description complexity. We extend this relationship also to Boltzmann entropy by establishing an asymptotic match between the two entropies. Despite the simplicity of unary structures, our arguments require the use of formula size games, Stirling's approximation and Chernoff bounds.
\end{abstract}

\section{Introduction}

This paper investigates the resources needed to 
define finite models with a unary 
relational vocabulary. 
While unary models are very simple, it turns out that proving limits on the formula sizes for defining them is non-trivial. 
Furthermore, unary models are important as they give a direct relational representation of Boolean data sets, consisting simply of data points and their properties---thereby providing one of the simplest data representation schemes available. In practice all tabular data can be discretized and modeled via a Boolean data set. This relates to applications in, e.g., explainability and compression.

%Unary structures are a relational representation of
%structures in the \emph{Boolean data model}, i.e., multisets of propositional assignments. Multisets of propositional assignments constitute one of the simplest data models available, consisting simply of data points with some properties. 

Given a logic $\mathcal{L}$ and a class $\mathcal{M}$ of models, the 
\emph{description complexity} $C(\MM)$ of a model $\MM$ is the minimum length of a 
formula $\varphi\in \mathcal{L}$ that defines $\MM$ with respect to $\mathcal{M}$. In the main scenario of this paper, $\mathcal{M}$ is the class of models with the same 
domain of a finite size $n$ and with the same unary vocabulary $\tau$. 
We mostly study the setting via first-order logic FO. However, as description complexity links to the
themes of \emph{compressibility} and
\emph{compression}, we also investigate the restricted languages $\FO_d$ where the quantifier rank of every formula is limited to a positive integer $d$. This will lead to dramatically shorter description lengths (cf. Section \ref{sec:upperbounds}) via a natural lossy compression phenomenon.

%This data representation corresponds exactly to relational models with unary relations and can be illustrated via Venn diagrams. 

%The main results of the paper concern 
%description complexities of unary models and links between description complexity and Shannon entropy.

We also investigate how the \emph{Shannon entropies} of unary structures are linked to their description complexities, the general trend being that higher entropy relates to higher description complexity. Shannon entropy is a well-known measure of intrinsic complexity, or randomness, from information theory.
The Shannon entropy of a probability distribution $\mathbb{P}: X \rightarrow [0,1]$ over a finite
set $X$ is
given by $- \sum_{x\in X} \mathbb{P}(x) 
\log_2 \mathbb{P}(x).$
A relational structure $\mathfrak{M}$ of size $n$ over a unary vocabulary $\tau$ naturally defines a probability distribution
over its domain. 
Indeed, let $T$ be the set of unary quantifier-free \emph{types} over $\tau$, i.e., subsets of $\tau$. 
%A point $a$ of a model $\mathfrak{M}$ \emph{realizes} a type $\pi$ if for every unary relation symbol $P$, $a$ belongs to the corresponding unary relation $P^\MM$ if and only if $P \in \pi$.
A point $a$ of a model $\mathfrak{\MM}$ \emph{realizes} a type $\pi \subseteq \tau$ if $\pi$ is the set of relation symbols corresponding to exactly those unary relations that contain the point $a$.
% if \(\pi\) consists of precisely those relation symbols for which the corresponding relations contain \(a\).
%If \(a\) belongs to the interpretation of \(P \in \tau\) precisely when \(P \in \pi\).
%Indeed, let $T$ be the set of 
%unary quantifier-free \emph{types} over $\tau$, i.e., 
%conjunctions $\varphi$ such that for each $P\in \tau$, 
%exactly one of the literals $P(x), \neg P(x)$ is a
%conjunct of $\varphi$ (and there are no other
%conjuncts). 
%To ensure there are exactly $2^{|\tau|}$ 
%such types, we also assume some standard ordering of
%conjuncts and bracketing for  types. 
Now 
%, letting $T$ be the set of all types over $\tau$, 
a $\tau$-model $\mathfrak{M}$ of size $n$
naturally defines the probability 
distribution $\mathbb{P}: T \rightarrow [0,1]$ 
such that $\mathbb{P}(\pi) = \frac{|\pi|}{n}$, where $|\pi|$ is the number of points of $\mathfrak{M}$ realizing the
type $\pi$. The Shannon entropy of $\mathfrak{M}$ is then naturally defined to be equal to the Shannon entropy of the distribution $\mathbb{P}:T\rightarrow [0,1]$.

While the Shannon entropy of $\mathfrak{M}$ gives an intrinsic measure of complexity (or randomness) of $\mathfrak{M}$, another entropy
measure may perhaps be easier to grasp intuitively. \emph{Boltzmann entropy} has its origins in statistical mechanics, and it was originally defined as $k\ln \Omega$, where $k$ is the Boltzmann
constant and $\Omega$ the number of \emph{microstates} of a system. In our setting, we follow \cite{stacspaper} and 
define Boltzmann entropy of a \emph{model class} $\mathcal{A}$ as $\log_2 |\mathcal{A}|$, thus dropping the Boltzmann constant $k$, using binary logarithms and associating models with microstates. Now, it is natural to then define the Boltzmann entropy of a model $\mathfrak{M}$ as $\log_2|\mathcal{M}|$,
where $\mathcal{M}$ is the isomorphism class of $\mathfrak{M}$ (recall here that in our setting, all models have the same
domain of size $n$, so $\mathcal{M}$ is finite). The reason why the Boltzmann entropy of $\mathfrak{M}$ is a reasonable measure of intrinsic complexity of $\mathfrak{M}$ is now easy to motivate. Firstly, consider a $\tau$-model $\mathfrak{M}_0$ of
size $n$ where each $P\in \tau$ is interpreted as the empty relation. This is a very simple model whose isomorphism class has size $1$ and the Boltzmann
entropy of $\mathfrak{M}_0$ is thus very low: $\log_2 1 = 0$.
On the other hand, models with the predicates in $\tau$ distributed in more disordered ways have larger isomorphism classes and thus greater Boltzmann entropies.

\subsection{Contributions}

%\subsection{Upper bounds}

Concerning upper bounds on description complexity, we show how to define unary structures via $\FO$-formulas that are linear in model size. This contrasts the standard quadratic formulas that use equalities for counting cardinalities in a naive way. 
We also give analogous formulas for $\FO_d$ with
quantifier rank at most $d$.
%
%
%
%\subsection{Lower bounds}
Concerning lower bounds,   
 we use formula size games to provide bounds with a worst case gap of a constant factor of 2 in relation to the upper bounds. This is done both for full $\FO$ and $\FO_d$.

For a random structure the upper and lower bounds in fact match up to a sublinear additive term. 
Using this, we show that---asymptotically---the expected description complexity of a random unary structure of size $n$ and over the vocabulary $\tau$ is exactly $3n/{2^{|\tau|}}$ .

%\subsection{Results on entropy}
We then turn our attention to entropy. 
 We show a close relationship between the Shannon entropy and Boltzmann entropy of a unary structure. We obtain related upper and lower bounds and thereby also establish the following asymptotic equivalence for \emph{every} sequence $\MM_n$ of models of increasing size $n$: 
$H_S(\MM_n) \sim \frac{1}{n}H_B(\MM_n).$
We note that a result bearing a resemblance to this one has been obtained in a slightly different framework in \cite{Kolmogorov1993}.

Finally, we relate the description complexity of a model to its entropy. We investigate the general picture of the relationship by giving upper and lower bounds on the description complexity of a model in terms of its entropy.
See Figure \ref{fig:entropy} for the case of $\FO$ and Figure \ref{fig:d-entropy} for $\FO_d$. 
The bounds allow us to \emph{exclude} a large portion of the (a priori) possible combinations of description complexity and entropy.
In particular, we see that models with very high entropy have higher description complexity than models with very low entropy.
Moreover, \emph{models with a very low entropy are guaranteed to have a reasonably low description complexity, while models with very high entropies must have a notable description complexity.} 

\subsection{Related work, techniques and applications}\label{ssec:relatedwork}

%\subsubsection{Entropy and Kolmogorov complexity}

\begin{comment}
Description complexity is conceptually related to Kolmogorov complexity \cite{LiV19}. Recall that the Kolmogorov complexity of a binary string $x$ 
%is intuitively speaking the size of the smallest program that outputs $x$. 
can be defined as the size of the smallest program which distinguishes $x$ from all other strings.
%To make the connection between description complexity and Kolmogorov complexity explicit, we note that such a program can be viewed as a formula, that \emph{describes} $x$ in some logic.
Here the program can be viewed as a formula that describes $x$ in some logic. 
%Keeping this in mind, 
The size of the smallest such program is now the description complexity of $x$ in this logic.

The connection between description complexity and Kolmogorov complexity can be found on many levels of expressive power. For example, if the logic that we use for classifying structures is Turing-complete in expressive power, then the description complexity of a structure is very closely related to its Kolmogorov complexity (for an example of a Turing-complete logic, see, e.g., \cite{turingcomp}). On the other hand, the description complexity of weaker logics can be associated with \emph{resource bounded} Kolmogorov complexity, where the Turing machines involved have, for example, bounded runtime or memory \cite{Balcazar1990}.
\end{comment}

%Concerning related work, as already mentioned,
Description complexity is conceptually related to Kolmogorov complexity, and it is also well known that entropy and Kolmogorov complexity are linked. Indeed, for computable distributions, Shannon entropy links to Kolmogorov complexity to within a constant. This is discussed, e.g., in \cite{vitanyi,grunwald,leung}. However, \cite{teixeira} shows that the general link fails for R\'{e}nyi and Tsallis entropies.  See, e.g., \cite{grunwald,leung,teixeira} for discussions on R\'{e}nyi and Tsallis entropies. 
%The first connection between logical \emph{formula length} and entropy has---to our knowledge---been obtained in \cite{stacspaper,stacspaper}, where expected Boltzmann entropy is shown to be asymptotically equivalent to description complexity. 
%There exist well known links between Kolmogorov complexity and Shannon entropy, see for example \cite{vitanyi}. 
%
%
% 
%While there is relatively little previous work in the intersection of logic and entropy, 

Concerning work in the intersection of logic and entropy, 
the
%our
recent article \cite{stacspaper} by Jaakkola et al. provides related
results for a graded
modal logic GMLU over Kripke-models with the universal accessibility relation.
%where counting modalities $\Diamond^{\geq k}$ are interpreted over the universal 
%Kripke relation, i.e., the full relation $W\times W$ when the domain of
%the model considered is $W$. The setting does not allow other binary relations.  
%The authors
%It is shown 
They show that the expected Boltzmann entropy of the equivalence classes of GMLU is asymptotically equivalent to the expected description complexity times the vocabulary size.
While \cite{stacspaper} concerns GMLU, the current paper studies (monadic) $\FO$. Because of the multi-variable nature of $\FO$, this leads to some major differences in the techniques required. The upper bound formulas of the current paper use some clever tricks that are not possible in the modal logic GMLU. Indeed, together with the results of \cite{stacspaper}, our upper bound formulas show that $\FO$ is more succinct than GMLU. Furthermore, the techniques used for the lower bounds for GMLU do not suffice for $\FO$, necessitating new arguments.
%The lower bound proof of the current paper utilizing formula size games also turns out very different for the same reason.
%Very 
Surprisingly, the relationship to entropy also turns out to be different. Indeed, in the case of GMLU, models with maximal entropy have maximal description complexity, while in the case of \(\FO\) this is no longer the case.
For proving bounds on formula sizes, we use \emph{formula size games} for $\FO$. Indeed, variants of standard Ehrenfeucht-Fra\"{i}ss\'{e} games would not suffice, as we need to deal with formula length, and thereby with all logical operators, including connectives. The formula size game that we use for $\FO$ is a slight modification of the game of Hella and V\"{a}\"{a}n\"{a}nen  \cite{HellaV15}.
The first formula size game, developed by Razborov in \cite{Razborov90}, dealt with propositional logic. A later variant of the game was defined by Adler and Immerman for $\mathrm{CTL}$ in \cite{AdlerI03}. In \cite{HellaV19} the formula size game for modal logic ML was used by Hella and Vilander to establish that bisimulation invariant FO is non-elementarily more succinct than ML. For a further example, we also mention the frame validity games of Balbiani et al. \cite{BalbianiFHI22}.
Recently, Fagin et al. in \cite{MultiStructuralGames1,MultiStructuralGames2} and Carmosino et al. in \cite{carmosino2024numberquantifiersneededdefine,CarmosinoFIKLS2023,Carmosino2024ParallelPS} have developed and used \emph{multi-structural games} to prove lower bounds on the \emph{number of quantifiers} that are needed for separating two structures in a given logic. In \cite{MultiStructuralGames2} they have also pointed out that strong lower bounds on the number of quantifiers would imply new lower bounds in circuit complexity.

%One should also mention the frame validity games of \cite{BalbianiFHI22}.

%While %finding 
%designing the games for $\plc^d$ is relatively straightforward and based directly on similar earlier systems, it is more challenging to use the game in a way suitable for our purposes.

%In addition to games, we also make use of combinatorial techniques for estimating model class sizes and description complexity. 
%These include
%inter alia, 
%Stirling's approximations and Chernoff bounds. 
%In particular, to obtain our results, we prove \emph{new estimates on $r$-associated Stirling numbers}, which may be of independent interest. 

%\subsubsection{Data compression and explainability}

%While the current paper focuses on theory, the notion of description complexity is relevant in many applications. One practically significant 
Description complexity is relevant in many applications, one
interesting link being data compression. It is natural to consider unary models $\mathfrak{M}$ as data sets to be compressed into corresponding FO-sentences. To give a \emph{simplified} example, let $\mathcal{M}$ be the class of 
models over the unary 
alphabet 
$\tau = \{P,Q\}$ and with domain $M = \{1, \dots , 10\}$. 
Let $\MM_1$ be the model
where $P^{\mathfrak{M}_1} = Q^{\mathfrak{M}_1} = M$
and $\MM_2$ be the model where $P^{\mathfrak{M}_2} = \{1,2,3\}$ and $Q^{\mathfrak{M}_2}$ is,
say, $ \{3,4,5,6,7\}$. 
Now, the simple formula $\forall x (P(x) \land Q(x))$ fully defines the model $\MM_1$ with respect to $\mathcal{M}$, while the model $\mathfrak{M}_2$ clearly requires a more complex formula. 
Suppose then that our models are 
represented as tabular Boolean
data, meaning that each model corresponds 
%in the natural way 
to a 0-1-matrix with ten rows (one row 
for each domain element $m \in M$) and two columns, one column for $P$ and another one for $Q$. In this framework, when using FO as a compression language, the
Boolean matrix for $\mathfrak{M}_1$ then compresses nicely into the
formula $\forall x (P(x) \land Q(x))$, while the 
matrix for $\mathfrak{M}_2$ compresses to a notably more complex formula. 
%
%
%
\begin{comment}
We can see from these simple examples that---as expected---some models can be compressed without loss via very short formulas, while other models require longer formulas to be fully described. In this case the partial description of $\MM_2$ is a form of \emph{lossy compression}. By increasing the counting depth $d$, the information loss decreases and we can ultimately, with a large enough $d$, describe models up to isomorphism. The size of the class of models that satisfy the compressed formula is now a natural measure of the information loss. Our results relate the length of the compressed formula to the size of the information loss and entropy notions. 
\end{comment}

%Perhaps the main reason why there is relatively little work in the intersection of logic and data compression relates to the fact that investigating formula length has mainly been very difficult. However, developments such as formula size games offer some tools to tackle some of the related problems. 

%While compressibility is a generally important ingredient in explainable artificial intelligence (XAI), 
Many of the technical goals in \emph{explainable artificial intelligence} (XAI) relate to compression \cite{sarkar}, often revolving around issues of compressing information given by 
probability distributions. It is natural to expect 
representations of distributions with very high values of Shannon entropy to be more difficult to compress than ones 
with very low values.
Concerning formula length, 
recent articles on XAI using minimum length
formulas of logics as 
explanations of longer specifications
%topic 
include, e.g.,\cite{explainability2,ExplainingMonotoneClassifiers,explainability1, jelia}, and numerous others. For work on using short Boolean formulas as general explanations of real-life data given in the form of unary relational structures (i.e., tabular Boolean data sets), see \cite{jelia}. In that paper, surprisingly short Boolean formulas are shown to give similar error rates to ones obtained by more sophisticated classifiers, e.g., neural
networks and naive Bayesian classifiers. 
%We note that in the framework of \cite{jelia}, explanations are short Boolean formulas that essentially encode information about probability distributions. Using short FO-formulas, while one would lose the simplicity and other benefits of Boolean formulas, it would be possible to more accurately encode infomation about probability distributions as. 
% 

Concerning further directions in 
explainability, 
minimum size descriptions $\psi$ of unary relational models $\mathfrak{M}$ can be useful for
finding explanations in the context of the \emph{special explainability problem} \cite{explainability1}.
The positive case of this problem amounts to finding formulas $\chi$ with a given bound $k$ on length such that $\mathfrak{M}\vDash \chi \vDash \varphi$, where $\varphi$ acts as a classifier. 
In this context, it often suffices to find a short interpolant $\chi$ such that $\psi\vDash  \chi \vDash  \varphi$, where $\psi$ is a minimum description of $\mathfrak{M}$. In applications, this latter task can often be more efficient than the first one, especially when $\psi$ is significantly smaller than $\mathfrak{M}$.
One way to ensure $\psi$ is short enough is to describe $\mathfrak{M}$ in a sufficiently incomplete way, such as with $\FO_d$ with small $d$.

Finally, in applications, it is typically easy to compute the Shannon entropy of structures, while 
description complexity and thereby issues relating to compressibility and explainability are \emph{much more difficult to determine}. Therefore, even a rough picture of the links between entropy and description complexity can be useful.

%Finally, it is worth stressing that there exists relatively little work bridging the gap between logic and  topics in information and coding theory. One of the main reasons for this is indeed that tools for measuring formula length in formal logic are, for the most part, relatively recent. 
%As computational logics offer a formal setting for comparing and classifying different kinds of description languages (i.e., different computational logics), we believe that novel links between the above mentioned fields can be fruitful.

The plan of the paper is as follows. After the preliminaries in Section \ref{sec:preliminaries}, we provide upper bounds for the description complexity of unary structures in Section \ref{sec:upperbounds}. In Section \ref{sec:lower-bounds} we establish related lower bounds using games. In Section \ref{sec:expected-description-complexity} we determine asymptotically the expected description complexity of a random unary structure. In Section \ref{sec:entropy} we give bounds on the relationship between entropy and description complexity. In Section \ref{sec:conclusion} we conclude.
%After the preliminaries in Section \ref{preliminaries}, we prove 
%crucial 
%lower bounds for description complexity in Section \ref{descriptioncomplexitysection} using games. In Section \ref{monotoneconnection} we prove a monotone connection between model class size and description complexity, and in Section \ref{phase} we investigate phase transitions of class size distributions by varying $n$ and $d$. 
%Section \ref{conclusion} concludes the paper.

\section{Preliminaries}\label{sec:preliminaries}

%We begin with basic definitions related to first-order logic.
Let $\tau = \{P_1, \dots, P_k\}$ be a monadic vocabulary and let $\Var = \{x_1, x_2, \dots\}$ be a countably infinite set of variables. The syntax of first-order logic $\FO[\tau]$ is generated by the grammar:
$
\varphi ::= x = y \mid P(x) \mid \neg \varphi \mid \varphi \lor \varphi \mid \varphi \land \varphi \mid \exists x \varphi \mid \forall x \varphi,
$
where $x,y \in \Var$ and $P \in \tau$. 
The \textbf{quantifier rank} of a formula $\varphi \in \FO[\tau]$ is the maximum number of nested quantifiers in the formula. We denote by $\FO_d[\tau]$ the fragment of $\FO[\tau]$ that only includes the formulas with quantifier rank at most $d$.
A formula $\varphi \in \FO[\tau]$ is in \textbf{negation normal form} if negations are only applied to \textbf{atomic formulas} $x=y$ or $P(x)$. We assume all formulas are in negation normal form and treat the notation $\neg \varphi$ as shorthand for the negation normal form formula obtained from $\varphi$ by pushing the negation to the level of atomic formulas.

The \textbf{size} of a formula $\varphi \in \FO[\tau]$ is defined as the number of atomic formulas, conjunctions, disjunctions and quantifiers in $\varphi$. Note that negations do not contribute to the size of $\varphi$. This choice together with using negation normal form means that positive and negative atomic information is treated as equal in terms of formula size. In line with this thinking, we will refer also to $x \neq y$ and $\neg P(x)$ as atomic formulas in the sequel.

A formula $\varphi \in \FO[\tau]$ is in \textbf{prenex normal form} if it is of the form
$Q_1x_1\dots Q_mx_m \psi,$
where $Q_i \in \{\exists, \forall\}$ for $i \in \{1, \dots, m\}$ and $\psi \in \FO[\tau]$ has no quantifiers. It is well-known that every $\FO$-formula can be transformed into an equivalent formula in prenex normal form which has the same size as the original formula.

A \textbf{$\tau$-model} is a tuple $\MM = (M, P_1^\MM, \dots, P_k^\MM)$, where $M = \{1, \dots, n\}$ and $P_i^\MM \subseteq M$ for $i \in \{1, \dots, k\}$. A model $\MM$ is a \textbf{model of size $n$} if $|M| = n$. A partial function $s : \Var \rightharpoonup M$ is called an \textbf{interpretation}. We also call pairs $(\MM, s)$ models and identify the pair $(\MM, \emptyset)$ with the model $\MM$. The truth relation $(\MM, s) \vDash \varphi$ is defined in the usual way for $\FO[\tau]$.

Let $\MM = (M, P_1^\MM, \dots, P_k^\MM)$ be a $\tau$-model of size $n$. We say that a formula $\varphi \in \FO[\tau]$ \textbf{defines} $\MM$ if for all $\tau$-models $\MM'$ of size $n$ we have $(\MM', \emptyset) \vDash \varphi$ iff $\MM'$ is isomorphic to $\MM$. As first-order logic cannot distinguish between isomorphic structures, we can in some sense identify the model $\MM$ with the class of models isomorphic to $\MM$. The \textbf{description complexity} $C(\MM)$ of $\MM$ is the size of the smallest formula in $\FO[\tau]$ that defines $\MM$. 

Note that our definition of description complexity concerns separating $\MM$ only from other models of the same size \(n\). Requiring separation from all other models would unduly emphasize the size of the model, making even very simple models have a high description complexity.
For example, the model $\MM = (M, P^{\MM})$ of size $n$, where \(P^{\MM} = M\), would already require a formula with size in the order of \(n\). In our setting, \(C(\MM) = 2\), because \(\MM\) is defined by the formula \(\forall x P(x)\).

A $\tau$-\textbf{type} $\pi$ is a subset of $\tau$. A point $a \in M$ \textbf{realizes} a $\tau$-type $\pi$ if for all $P \in \tau$ we have $a \in P^{\MM}$ iff $P \in \pi$. We let $|\pi|_\MM$ denote the number of points in $\MM$ realizing $\pi$. We often omit the subscript when the model is clear from the context. Note that two $\tau$-models $\MM$ and $\MM'$ are isomorphic iff each type is realized in the same number of points in both models.

We also consider more coarse ways to divide models into classes than isomorphism. For each positive integer $d$ we can define an equivalence relation $\equiv_d$ over $\tau$-models of size $n$ as follows. Given two $\tau$-models $\MM$ and $\MM'$ of size $n$, we define that $\MM \equiv_d \MM'$ iff for each $\tau$-type $\pi$ with $|\pi|_\MM < d$, we have that $|\pi|_\MM = |\pi|_{\MM'}$. In other words, $\MM \equiv_d \MM'$ iff each type that is realized in less than $d$ points in $\MM$ is realized in the same number of points in both models. It is easy to show that $\MM \equiv_d \MM'$ iff they satisfy the same sentences of $\FO_d[\tau]$. The \textbf{$d$-description complexity} $C_d(\MM)$ of a $\tau$-model $\MM$ is the size of the smallest $\FO_d[\tau]$-formula that defines the equivalence class of $\MM$ in $\equiv_d$.

To characterize model classes, we use tuples with $t = 2^{|\tau|}$ numbers. For an isomorphism class, the tuple is simply $(|\pi_1|, \dots, |\pi_t|).$ For an equivalence class $\mathcal{M}$ of $\equiv_d$, we only use numbers up to $d$. For a tuple $\overline{m} = (m_1, \dots, m_t)$, if $m_i = d$, then there are at least $d$ realizing points of type $\pi_i$ in models of the class $\mathcal{M}$. If $m_i < d$, then each model has exactly $m_i$ points realizing the type $\pi_i$. The notation $\mathcal{M}_{\overline{m}}$ refers to classes of $\equiv_d$ via these tuples. The tuples that correspond to some class of $\equiv_d$ are characterized by the conditions $m_i \leq d$ for $i \in \{1, \dots, t\}$, $\sum_{i = 1}^t m_i \leq n$ and if $\sum_{i = 1}^t m_i < n$, then $m_j = d$ for some $j \in \{1, \dots, t\}$. If $\sum_{i = 1}^t m_i = n$, then $\mathcal{M}_{\overline{m}}$ is an isomorphism class.

Since $\tau$-types partition the points of a $\tau$-model $\MM$, we may consider a natural probability distribution over the types in $\MM$. The probability $p_\pi$ of a type $\pi$ is simply $|\pi|/n$, that is, the probability of hitting a point of type $\pi$ when selecting a point from $\MM$ randomly. The \textbf{Shannon entropy} of $\MM$ is the quantity
%
%\[
%
$
    H_S(\MM) := \sum_{i = 1}^t -p_{\pi_i}\log(p_{\pi_i})= \sum_{i = 1}^t - \frac{|\pi_i|}{n} \log\big(\frac{|\pi_i|}{n}\big).  
$ 
%\]
%
Shannon entropy is an information theoretic way of measuring randomness of probability distributions. Uniform distributions have maximal Shannon entropy, as the uncertainty of the outcome of choosing a random point is maximized. Conversely, for a distribution that places all of the probability mass on a single event, Shannon entropy is zero. Hence, a model realizing each type the same number of times (or as close as possible) has maximal Shannon entropy, while for a model that realizes only a single type Shannon entropy is zero.

Another way to define entropy of a model $\MM$ uses the model class $\MM$ belongs to. Given an equivalence relation $\equiv$ over models of size $n$ (and thus
domain $\{1,\dots , n\}$),
%$n$-element models
 the \textbf{Boltzmann entropy} of $\MM$ with respect to $\equiv$ is
$
    H_B(\MM) := \log(|\mathcal{M}|),
$
where $\mathcal{M}$ is the equivalence class of $\MM$. In this paper the equivalence relation $\equiv$ is either isomorphism in the case of full $\FO$ or $\equiv_d$ for $\FO_d$. For isomorphism, we write $H_B(\MM)$ and for $\equiv_d$ we write $H_B^d(\MM)$.

Boltzmann entropy originates from statistical mechanics, where it measures the randomness of a macrostate (= a model class) via the number of microstates (= models) that correspond to it. The idea is that a larger macrostate is ``more random'' (or ``less specific'') since it is more likely to be hit by a random selection.
We show in Section \ref{sec:entropy} that
$H_S(\MM) \sim \frac{1}{n} H_B(\MM),$
where $n$ is the size of the domain of $\MM$. Thus the two notions of entropy are asymptotically equivalent up to normalization. This shows that both entropies indeed measure the randomness of a model from different points of view.

\section{Upper bound formulas}\label{sec:upperbounds}

In this section we define arbitrary $\tau$-models via formulas of size linear in the size of the model. Recall that defining a model means separating it from all non-isomorphic models with the same domain size. To see why linear size formulas are quite succinct, note that the following naive formula
\[
    \bigwedge_{\ell = 1}^{2^{|\tau|}} \exists x_1 \dots \exists x_{|\pi_\ell|} \bigg(\bigwedge_{i = 1}^{|\pi_\ell|} \pi(x_i) \land  \bigwedge_{j = i + 1}^{|\pi_\ell|} x_i \neq x_j\bigg),
\]
which expresses that for each $1 \leq \ell \leq 2^{|\tau|}$ the type $\pi_\ell$ is realized by at least $|\pi_\ell|$ distinct points, is of quadratic size in the size $n$ of the model.

For clean results on formula size, we define a constant $c_\tau := 15|\tau|2^{|\tau|}$.

\begin{theorem}\label{thm:upper-bound}
    Let $\MM$ be a model of size $n$. Let $T = \{\pi_1, \dots, \pi_\ell\}$ be the types realized in $\MM$, enumerated in ascending order of numbers of realizing points. Now we have 
    \[
        C(\MM) \leq \min(3|\pi_\ell|+c_\tau, 6|\pi_{\ell-1}| + c_\tau).
    \]
\end{theorem} 
\begin{proof}
    %We obtain two different upper bound formulas. Due to lack of space, we only give one of them in full here; see \ref{app:formulas} for details on the second formula.
We begin with an auxiliary formula we use extensively below. For a type $\pi$ and $x \in \Var$, we define 
\[
\pi(x) := \bigwedge\limits_{P \in \pi} P(x) \land \bigwedge\limits_{P \notin \pi} \neg P(x).
\]
The formula $\pi(x)$ states that the point $s(x)$ realizes the type $\pi$.

Let $T = \{\pi_1, \dots, \pi_\ell\}$ be a set of $\tau$-types and let $\overline{m}$ be a sequence of $r \leq \ell$ positive integers with $0 < m_1 \leq  \dots \leq m_r$. Let $\MM$ be a model, where exactly the types in $T$ are realized. We will make sure of this with a separate formula later. The formula $\varphi(T, \overline{m})$ below is satisfied by such a model $\MM$ if and only if for every $i \in \{1, \dots, r\}$, the model $\MM$ has \emph{at least} $m_i$ points that realize the type $\pi_i$. Note that we do not assert anything about the types $\pi_{r+1}, \dots, \pi_\ell$, but we still need to mention them in the formula. We define
\begin{align*}
 &\psi_{m_r} := y \neq x_{m_r-1} \land \bigvee\limits_{\substack{j \in \{1, \dots, r\} \\ m_j = m_r}} (\pi_j(x_1) \land \pi_j(y)) \\
 &\psi_i := y \neq x_{i-1} \land \psi_{i+1} \text{, if $m_j \neq i$ for all $j \in \{1, \dots, r\}$, and} \\
 &\psi_i := y \neq x_{i-1} \land (\bigvee\limits_{\substack{j \in \{1, \dots, r\} \\ m_j = i}} (\pi_j(x_1) \land \pi_j(y)) \lor \psi_{i+1}) \text{, otherwise.} \\
 &\psi_1 := \psi_2 \text{, if $m_j \neq 1$ for all $j \in \{1, \dots, r\}$, and} \\
 &\psi_1 := \bigvee\limits_{\substack{j \in \{1, \dots, r\} \\ m_j = 1}} \pi_j(x_1)  \lor \psi_2 \text{, otherwise.} \\
 &\varphi(T, \overline{m}) := \forall x_1 \dots \forall x_{m_r-1}\exists y (\bigvee\limits_{j \in \{r+1, \dots, \ell\}} \pi_j(x_1) \lor \psi_1)
\end{align*}
We adopt the notation $k = |\tau|$ and compute the size of $\varphi(T, \overline{m})$. The formula has $m_r$ quantifiers. For each type $\pi \in T$, there are at most two occurrences of the subformula $\pi(x)$ (with different variables $x$). Each subformula $\pi(x)$ contains $k$ atomic formulas. Thus there are at most $2k|T|$ atomic formulas of the form $P(x)$ or $\neg P(x)$. Each inequality $y \neq x_i$ for $1 \leq i \leq m_r -1$ occurs exactly once, so there are $m_r-1$ atomic formulas that are equalities or inequalities. Finally we multiply the number of atomic formulas by two and subtract one to also account for the binary connectives. The size of $\varphi(T, \overline{m})$ is thus at most
\[
    m_r +  2(m_r-1 + 2k|T|)-1 
    = 3m_r + 4k|T| - 3. 
\]

We proceed with an explanation of how the formula $\varphi(T, \overline{m})$ works. We assume that precisely the types in $T$ are realized in the model $\MM$ to be evaluated, so we know that the first universal variable $x_1$ is always attached to a point that realizes one of the types in $T$. The formula first checks if $x_1$ realizes one of the types $\pi_{r+1}, \dots, \pi_\ell$ that we wish to ignore. The recursion then handles the rest of the types, starting with the smallest ones. If the type $\pi_j$ of $x_1$ has $m_j = 1$, nothing further is stated as we already know the type is realized in $\MM$ by our assumption. 
%If the type $\pi_j$ of $x_1$ has $m_j = 2$, the formula insists that the existential variable $y$ is attached to a different point than $x_1$ and both points are of the type $\pi_j$. Thus the formula insists that there are at least two points of the type $\pi_j$. 

Now, consider a type $\pi_j$ with, say, $m_j = 5$. Up to the subformula $\psi_5$, the recursion of our formula has insisted that $y \neq x_i$ for $i \in \{1,2,3,4\}$. Note that the formula does not contain any atomic formulas $x_{i_1} \neq x_{i_2}$. The crucial point is that since the variables $x_1, \dots, x_4$ are universally quantified, the existence of $y$ must hold also in the case, where $x_1, \dots x_4$ happen to all be different points of the same type $\pi_j$. If the evaluated model $\MM$ has at least 5 points that realize $\pi_j$, then the formula holds as another point $y$ that realizes $\pi_j$ can be found. If, however, $\MM$ has only 4 points that realize $\pi_j$, then one of the universally quantified tuples includes precisely those 4 points and another $y$ of the same type cannot be found.

Let $T$, $\overline{m}$ and $\MM$ be as above.
We define another formula $\chi(T, \overline{m})$ below. Now the model $\MM$ satisfies $\chi(T, \overline{m})$ if and only if for every $i \in \{1, \dots, r\}$, the model $\MM$ has \emph{at most} $m_i$ points that realize the type $\pi_i$. We again do not assert anything about the types $\pi_j$ with no corresponding $m_j$.
\begin{align*}
    \theta_{m_r} &:= y = x_{m_r} \lor \bigvee\limits_{\substack{j \in \{1, \dots, r\} \\ m_j = m_r}} (\pi_j(x_1) \land \neg \pi_j(y)) \\
    \theta_i &:= y = x_{i} \lor \theta_{i+1} \text{, if $m_j \neq i$ for all $j \in \{1, \dots, r\}$, and } \\
    \theta_i &:= y = x_{i} \lor (\bigvee\limits_{\substack{j \in \{1, \dots, r\} \\ m_j = i}} (\pi_j(x_1) \land \neg \pi_j(y)) \lor (\bigwedge\limits_{\substack{j \in \{1, \dots, r\} \\ m_j = i}} \neg \pi_j(x_1) \land \theta_{i+1}),  \\ &\text{otherwise.} \\
    \chi(T, \overline{m}) &:= \forall x_1 \exists x_2 \dots \exists x_{m_r} \forall y (\bigvee\limits_{j \in \{r+1, \dots, \ell\}} \pi_j(x_1) \lor \theta_1)
\end{align*}
We compute the size of $\chi(T, \overline{m})$. The formula has $m_r +1$ quantifiers. For each type $\pi$, the subformula $\pi(x)$ occurs at most three times and for at least one type with $|\pi| = m_r$, only two times. This results in $3k|T|-k$ atomic formulas of the form $P(x)$ or $\neg P(x)$. For the equalities and inequalities, each equality $y = x_i$ for $1 \leq i \leq m_r$ occurs exactly once, for a total of $m_r$ such atomic formulas. Accounting for the binary connectives, the size of $\chi(T, \overline{m})$ is thus at most
\begin{align*}
    &m_r+1+ 2(m_r+ 3k|T|-k)-1 
    =\  3m_r + 6k|T| -2k.
\end{align*}

We again explain how the above formula works. Note that directly taking the negation of the formula $\varphi(T, \overline{m})$ would not work as we are dealing with all types at once. We instead again start with a universally quantified variable $x_1$ that is attached to a point realizing a type $\pi_j \in T$. We first check if $\pi_j$ is one of the types we can safely ignore. Assume then that $m_j = 5$. The existentially quantified variables $x_2, \dots, x_5$ are then chosen to be of the same type $\pi_j$ as $x_1$ in such a way that every point of the type $\pi_j$ has at least one $x_i$ attached to it. Since $m_j = 5$, the first step of the recursion insists that either $y$ is the same as $x_1$ or the recursion continues. When the recursion arrives at $\theta_5$, we cannot go any further, as to continue, we would need $m_j \neq 5$. We are instead left with the two options of either $y = x_5$ or $y$ realizes a different type than $x_1$. This amounts to saying that there are no more than 5 points that realize the type $\pi_j$. 

The crucial point of the formula $\chi(T, \overline{m})$ is that the first universally quantified variable $x_1$ allows us to use the same existential quantifiers to count all types at once. To ensure that we do not require all of the types to be the same size, we restrict the type realized by $x_1$ before continuing with the recursion.

We proceed to define our complete upper bound formulas that define isomorphism classes of models. Let $\MM$ be a $\tau$-model with domain $M = \{1, \dots, n\}$. Let $T = \{\pi_1, \dots, \pi_\ell\}$ be the set of $\tau$-types realized in $\MM$ and let $\overline{m} = (|\pi_1|, \dots, |\pi_\ell|)$. Assume further that $\overline{m}$ is increasing. We use the formulas $\varphi(T, \overline{m})$ and $\chi(T, \overline{m})$ to separate the model $\MM$ from all other models with the same domain size. The first formula $\varphi(\MM)$ is based on bounding the size of every type in $T$ from below.
\begin{align*}
    \varphi(\MM) := \bigwedge\limits_{i = 1}^\ell \exists x \, \pi_i(x) \land \forall x \bigvee\limits_{i = 1}^\ell \pi_i(x) 
 \land \varphi(T, \overline{m})
\end{align*}
In addition to the size of $\varphi(T, \overline{m})$ computed above, $\varphi(\MM)$ includes $|T|+1$ quantifiers and two occurrences of $\pi(x)$ for each type $\pi$, resulting in $2k|T|$ atomic formulas. Accounting for the added binary connectives, the size of $\varphi(\MM)$ is thus at most
\begin{align*} 
|T|+1+2 \cdot 2k|T|+3|\pi_\ell|+4k|T|-3
= 3|\pi_\ell|+8k|T|+|T|-2. 
\end{align*}

Our second formula $\psi(\MM)$ avoids counting the type $\pi_\ell$ with the most realizing points by bounding the size of all other types from above and from below. For this formula we denote by $\overline{m} \setminus |\pi_\ell|$ the sequence $(|\pi_1|, \dots, |\pi_{\ell-1}|).$ We define
\begin{align*}
    \psi(\MM) := &\bigwedge\limits_{i = 1}^\ell \exists x \, \pi_i(x) \land \forall x \bigvee\limits_{i = 1}^\ell \pi_i(x) \land \varphi(T, \overline{m} \setminus |\pi_\ell|) \land \chi(T, \overline{m} \setminus |\pi_\ell|).
\end{align*}
The numbers of new quantifiers and atomic formulas are the same as for $\varphi(\MM)$. Accounting for the binary connectives, including the one connecting ${\varphi(T, \overline{m} \setminus |\pi_\ell|)}$ and $\chi(T, \overline{m} \setminus |\pi_\ell|)$, the size of $\psi(\MM)$ is now at most
\begin{align*}
&|T|+1+2(k|T|+k|T|)+3|\pi_{\ell-1}|+4k|T|-3+3|\pi_{\ell-1}|+6k|T|-2k+1 \\
=\  &6|\pi_{\ell-1}|+14k|T|+|T|-2k-1.
\end{align*}

For cleaner statements on formula size, we define 
\[
c_\tau := 15|\tau|2^{|\tau|}.
\]
As our main interest is asymptotic behaviour with respect to $n$, we view $c_\tau$ as a constant. We can see from the above that the size of $\varphi(\MM)$ is less than $3|\pi_\ell|+c_\tau$, while the size of $\psi(\MM)$ is less than $6|\pi_{\ell-1}|+c_\tau$.

While both of the above formulas work for any model $\MM$, their size depends on the distribution of types in $\MM$. The minimum of the sizes of $\varphi(\MM)$ and $\psi(\MM)$ gives us the upper bound of the claim for the description complexity of the model $\MM$. 
\end{proof}

%As a corollary, we get an upper bound on the description %complexity of an arbitrary model.

\begin{corollary}\label{cor:maximum-description-complexity}
    Let $\MM$ be a model of size $n$. Now
    $
        C(\MM) \leq 2n + c_\tau.
    $
\end{corollary}
\begin{proof}
    A model $\MM$ corresponding to the tuple $(0, \dots, 0, n/3, 2n/3)$ maximises the value of the expression $\min(3|\pi_\ell|+c_\tau, 6|\pi_{\ell-1}| + c_\tau)$, getting the value $2n + c_\tau$.
    %Consider a model $\MM$, where one type has $2n/3$ realizing points, another type has $n/3$ realizing points and the rest of the types are not realized. The expression $\min(3|\pi_\ell|+c_\tau, 6|\pi_{\ell-1}| + c_\tau)$ given in the above theorem evaluates to $2n + c_\tau$ on $\MM$. On the other hand, for any $\MM'$ with a distribution of types different from $\MM$, the largest type is smaller than $2n/3$ or the second largest is smaller than $n/3$, thus lowering the value of the expression.
\end{proof}

We now consider defining equivalence classes of $\equiv_d$. Recall that an equivalence class of $\equiv_d$ corresponds to a tuple $\overline{m} = (m_1, \dots, m_t)$, where $t = 2^{|\tau|}$, $m_i \leq d$ for all $i \in \{1, \dots, t\}$, $\sum_{i = 1}^t m_i \leq n$ and if $\sum_{i = 1}^t m_i < n$, then $m_j = d$ for some $j \in \{1, \dots, t\}$.

\begin{theorem}\label{thm:d-upper-bound}
    Let $\MM$ be a $\tau$-model of size $n$. Let $\mathcal{M}_{\overline{m}}$ be the equivalence class of $\MM$ in $\equiv_d$, where $\overline{m} = (m_1, \dots, m_t)$ is the corresponding tuple with the numbers in ascending order. Let $m_r$ be the highest number in $\overline{m}$ below $d$. Now
    $
        C_d(\MM) \leq 3d+3m_r+c_\tau.
    $
    Additionally, if $m_{t-1} < d$, then
    $
        C_d(\MM) \leq 6m_{t-1}+c_\tau.
    $
\end{theorem}
\begin{proof}
    Let $\overline{m} = (m_1, \dots, m_t)$ be a tuple corresponding to a class of $\equiv_d$, ordered in the following way. The first numbers $m_1, \dots, m_r$ are the ones greater than $0$ and smaller than $d$ in ascending order. The numbers $m_{r+1}, \dots, m_\ell$ are all equal to $d$, and finally the numbers $m_{\ell+1}, \dots, m_t$ are all equal to 0. 

Using this order for the types, the set $T = \{\pi_1, \dots, \pi_\ell\}$ is now the set of types realized in models of the class and the first $r$ types are each realized exactly $m_i < d$ times. This is in line with the notation of the formulas for full $\FO$ above.

We utilize the subformulas $\varphi(T, \overline{m})$ and $\chi(T, \overline{m})$ defined in the proof of Theorem \ref{thm:upper-bound}. Our first formula works for any $\overline{m}$. The formula states that each type $\pi_j$ is realized at least $m_j$ times and furthermore, the ones with $m_j < d$ are realized at most $m_j$ times.
\begin{align*}
    \varphi_d(\overline{m}) := &\bigwedge\limits_{i = 1}^\ell \exists x \, \pi_i(x) \land \forall x \bigvee\limits_{i = 1}^\ell \pi_i(x) \land \varphi(T, (m_1, \dots, m_\ell)) \land \chi(T, (m_1, \dots, m_r))
\end{align*}
In the same way as for $\psi(\MM)$ in the proof of Theorem \ref{thm:upper-bound}, the size of $\varphi_d(\overline{m})$ is at most
\begin{align*}
&|T|+1+2(k|T|+k|T|)+3d+4k|T|-3+3m_r+6k|T|-2k+1 \\
=\  &3d+3m_r+14k|T|+|T|-2k-1 \leq 3d+3m_r+c_\tau.
\end{align*} 

Our second formula is only for the special case, where there is exactly one $m_j$ equal to $d$. In this case, as with full $\FO$, we can avoid counting the type with the most realizing points. The rest of the types $\pi_j$ have $m_j < d$ and the formula states that each $\pi_j$ is realized at least and at most $m_j$ times. 
\begin{align*}
    \psi_d(\overline{m}) := &\bigwedge\limits_{i = 1}^\ell \exists x \, \pi_i(x) \land \forall x \bigvee\limits_{i = 1}^\ell \pi_i(x) \land \varphi(T, (m_1, \dots, m_r)) \land \chi(T, (m_1, \dots, m_r))
\end{align*}
Again in the same way as for $\psi(\MM)$ in the proof of Theorem \ref{thm:upper-bound}, the size of $\psi_d(\overline{m})$ is at most
\begin{align*}
&|T|+1+2(k|T|+k|T|)+3m_r+4k|T|-3+3m_r+6k|T|-2k+1 \\
=\  &6m_r+14k|T|+|T|-2k-1 \leq 6m_r+c_\tau.
\end{align*} 

The upper bounds of the claim follow.
    %We use the same subformulas from Theorem \ref{thm:upper-bound} to obtain two linear size formulas. See \ref{app:d-formulas} for the full formulas. The first formula of size $3d+3m_r+c_\tau$ works for any tuple $\overline{m}$ and states that each type $\pi_i$ has exactly $m_i$ points if $m_i < d$ and at least $d$ points if $m_i = d$. The second formula of size $6m_{t-1}+c_\tau$ states that each type $\pi_i$ with $i \neq t$ has exactly $m_i$ points and works only if all types except possibly $\pi_t$ have less than $d$ points. 
\end{proof}

Note that since $m_r < d$, we have $6m_r < 3d + 3m_r$ so the bound for the special case is tighter than the general one. While we must use the more general bound for any $\overline{m}$ with at least two instances of $d$, the tighter bound is significantly better for small classes with only one instance of $d$ in their tuple. For example, the class with the tuple $(0, \dots, 0, 1, d)$ gets an upper bound of $6+c_\tau$ regardless of the number $d$. At the other extreme, the class with the tuple $(0, \dots, 0, d-1, d, d)$ gets an upper bound of $3d + 3(d-1) + c_\tau = 6d - 3 + c_\tau$.

We again directly obtain a global upper bound on description complexity.

\begin{corollary}
    Let $\MM$ be a $\tau$-model of size $n$. Now
    $
        C_d(\MM) \leq 6d-3+c_\tau.
    $
\end{corollary}

\section{Lower bounds via formula size games}\label{sec:lower-bounds}

In this section, we show lower bounds that match the upper bounds of Section \ref{sec:upperbounds} up to a factor of 2. We use the formula size game for first-order logic defined in \cite{HellaV15}.  We modify the game slightly to correspond to formulas in prenex normal form as this form does not affect the size of the formula. In addition, we introduce a second resource parameter $q$ that corresponds to the number of quantifiers in the separating formula. The game consists of two phases: a quantifier phase, where only $\exists$-moves and $\forall$-moves can be made by S, and an atomic phase, where only $\lor$-moves, $\land$-moves and atomic moves can be made. Before the definition of the game, we define some notation.

Let $\AAA$ be a set of $\tau$-models and let $\varphi \in \FO[\tau]$. We denote $\AAA \vDash \varphi$ to mean $(\MM, s) \vDash \varphi$ for all $(\MM, s) \in \AAA$. Similarly, we denote $\AAA \vDash \neg \varphi$ to mean $(\MM,s) \nvDash \varphi$ for all $(\MM, s) \in \AAA$.

For an interpretation $s$, a point $a \in M$ and a variable $x \in \Var$, we denote by $s[a/x]$ the interpretation $s'$ such that $s'(x) = a$ and $s'(y) = s(y)$ for all $y \in \mathrm{dom}(s)$, $y \neq x$. Let $\AAA$ be a set of $\tau$-models with the same domain $M$ and let $f: \AAA \to M$ be a function. We denote 
by
$\AAA[f/x]$ the set $\{(\MM, s[f(\MM, s)/x]) \mid (\MM, s) \in \AAA\}.$
Intuitively, the function $f$ gives the new interpretation of the variable $x$ for each model $(\MM, s) \in \AAA$. Additionally, we denote 
$
\AAA[M/x] := \{(\MM, s[a/x]) \mid (\MM, s) \in \AAA,\  a \in M\}.
$
Here the variable $x$ is given all possible interpretations, usually leading to a larger set of models. We next define the game.

Let $\AAA_0$ and $\BB_0$ be sets of $\tau$-models and let $r_0, q_0 \in \mathbb{N}$ with $r_0 > q_0$. The FO \textbf{prenex formula size game} $\game(r_0, q_0, \AAA_0, \BB_0)$ has two players: Samson (S) and Delilah (D). Positions of the game are of the form $(r, q, \AAA, \BB)$, where $r,q \in \mathbb{N}$ and $\AAA$ and $\BB$ are sets of $\tau$-models. The starting position is $(r_0, q_0, \AAA_0, \BB_0)$. In a position $(r, q, \AAA, \BB)$, if $r = 0$, then the game ends and D wins. Otherwise, if $q > 0$, the game is said to be in the \textbf{quantifier phase} and S can choose from the following three moves:
\begin{itemize}
    \item \textbf{$\exists$-move:} S chooses $f : \AAA \to M$ and $x_i \in \Var$. The new position is \\ $(r-1, q-1, \AAA[f/x_i], \BB[M/x_i])$. 
    \item \textbf{$\forall$-move:} The same as the $\exists$-move with the roles of $\AAA$ and $\BB$ switched.
    \item \textbf{Phase change:} S moves on to the atomic phase and the new position is $(r, 0, \AAA, \BB)$.
\end{itemize}
In a position $(r, q, \AAA, \BB)$, if $q = 0$, the game is said to be in the \textbf{atomic phase} and S can choose from the following three moves:
\begin{itemize}
    \item \textbf{$\land$-move:} S chooses $r_1, r_2 \in \mathbb{N}$ and $\BB_1, \BB_2 \subseteq \BB$ such that $r_1 + r_2 + 1 = r$ and $\BB_1 \cup \BB_2 = \BB$. Then D chooses the next position from the options $(r_1, 0, \AAA, \BB_1)$ and $(r_2, 0, \AAA, \BB_2)$.
    \item \textbf{$\lor$-move:} The same as the $\land$-move with the roles of $\AAA$ and $\BB$ switched.
    \item \textbf{Atomic move:} S chooses an atomic formula $\alpha$. The game ends. If $\AAA \vDash \alpha$ and $\BB \vDash \neg \alpha$, then S wins. Otherwise, D wins. 
\end{itemize}

The prenex formula size game characterizes separation of model classes with formulas of limited size in the following way.

\begin{theorem}\label{thm:gameworks}
    Let $\AAA_0$ and $\BB_0$ be sets of $\tau$-models and let $r_0, q_0 \in \mathbb{N}$ with $r_0 > q_0$. The following are equivalent
    \begin{enumerate}
        \item S has a winning strategy in the game $\game(r_0, q_0, \AAA_0, \BB_0)$,
        \item there is an $\FO[\tau]$-formula $\varphi$ in prenex normal form with size at most $r_0$ and at most $q_0$ quantifiers such that $\AAA_0 \vDash \varphi$ and $\BB_0 \vDash \neg \varphi$,
        \item there is an $\FO[\tau]$-formula $\varphi$ with size at most $r_0$ and at most $q_0$ quantifiers such that $\AAA_0 \vDash \varphi$ and $\BB_0 \vDash \neg \varphi$.
    \end{enumerate}
\end{theorem}
\begin{proof}
    For the simple inductive proof on how the game works, see \cite{HellaV15}. The slight modifications of the separate parameter $q$ for quantifiers and prenex normal form do not change the proof in any meaningful way so we omit it.  For the equivalence between the second and third item, note that transforming a formula into prenex form and renaming variables as needed, does not increase its size in full FO with no restrictions on, say, the number of variables.
\end{proof}

\begin{comment}
Let $\MM$ be a $\tau$-model with domain $M = \{1, \dots, n\}$. Let $\pi_1$ and $\pi_2$ be two types realized in $\MM$ such that $|\pi_2| \geq |\pi_1|$. We use the formula size game to show a lower bound of the order $3|\pi_1|$ for the description complexity of $\MM$.

Let $\MM'$ be the model obtained from $\MM$ by changing the type of one point from $\pi_1$ to $\pi_2$. We define $\AAA_0 = \{(\MM, \emptyset)\}$ and $\BB_0 = \{(\MM', \emptyset)\}$. For convenience, we denote $m = |\pi_1|_\MM$ below. We will show that separating the sets $\AAA_0$ and $\BB_0$ requires a formula of size at least $3m-3$.
\end{comment}

We take a moment to build some intuition on the formula size game. The role of player S is to show that the model sets $\AAA_0$ and $\BB_0$ can be separated by some $\FO$ formula with restrictions on size and number of quantifiers. To achieve this, S starts building the supposedly separating formula, starting from the quantifiers. 

Each move of the game corresponds to an operator or atomic formula. When making a move, S makes choices for each model that reflect how that particular model is going to satisfy the formula, in the case of models in $\AAA$, or not satisfy it, in the case of models in $\BB$. For example, for an $\exists$-move, S must choose for each model in $\AAA$ the point to quantify. This is done via the function $f$. For a $\land$-move, S chooses for each model in $\BB$ one of the conjuncts, asserting that the model will not satisfy that conjunct. 

The resources $r_0$ and $q_0$ restrict the moves of S. He can only make at most $q_0$ quantifier moves in the quantifier phase of the game. The resource $r_0$ limits the size of the entire separating formula, including the quantifiers. In the atomic phase, for $\land$-moves, S must divide the remaining resource $r$ between the two conjuncts. It is then the role of D to choose the conjunct she thinks cannot be completed in such a way that the models present are separated. Once D has chosen a conjunct, the other conjunct not chosen is discarded for the rest of the game. Thus, the entire separating formula need not be constructed during the game.

We move on to our lower bounds. Let $\MM$ be a $\tau$-model with domain $M = \{1, \dots, n\}$ and let $T = \{\pi_1, \dots, \pi_\ell\}$ be the types realized in $\MM$, enumerated in ascending order of numbers of realizing points, like in the previous section. We use the formula size game to show a lower bound of the order $3|\pi_{\ell-1}|$ for the description complexity of $\MM$.

Let $\MM'$ be the model obtained from $\MM$ by changing the type of one point from $\pi_{\ell-1}$ to $\pi_\ell$. We define $\AAA_0 = \{(\MM, \emptyset)\}$ and $\BB_0 = \{(\MM', \emptyset)\}$. 
%For convenience, we denote $m = |\pi_{\ell-1}|$ below. 
We will show that separating the sets $\AAA_0$ and $\BB_0$ requires a formula of size at least $3|\pi_{\ell-1}|-3$.
We begin with an easy lemma on the number of quantifiers required to separate $\AAA_0$ from $\BB_0$. 
\begin{lemma}\label{lem:quantifiers}
    If $\varphi$ separates $\AAA_0$ from $\BB_0$, then $\varphi$ has at least $|\pi_{\ell-1}|$ quantifiers.
\end{lemma}
\begin{proof}
    Let $r_0 > |\pi_{\ell-1}|-1$. We show that D has a winning strategy for the formula size game $\game(r_0, |\pi_{\ell-1}|-1, \AAA_0, \BB_0)$. By Theorem \ref{thm:gameworks}, this proves the claim.

    We show that in any position of such a game, there is a pair $(\MM, s) \in \AAA$ and $(\MM', s') \in \BB$ of models that cannot be separated by any atomic formula. At the starting position, the single models in $\AAA_0$ and $\BB_0$ are such a pair as no variables have been quantified. We proceed to show that D can maintain this pair of models through any move of S. We only treat one of each pair of dual moves as the other is handled the same way.

    $\exists$-move: S chooses a function $f : \AAA \to M$. We focus on the point $a = f(\MM, s)$ chosen for the model $(\MM, s) \in \AAA$. On the other side, copies of $(\MM', s') \in \BB$ are generated for each point $b \in M$, but we restrict attention to only one as follows. If there is a previously quantified variable $x$ with $s(x) = a$, then we choose $b = s'(x)$. Otherwise we choose a new point $b$ of the same type as $a$. If the type of $a$ is $\pi_i$ with $i < \ell-1$, then $\MM$ and $\MM'$ have the same points of type $\pi_i$ so we may choose $b = a$. If $i \in \{\ell-1, \ell\}$, then both $\MM$ and $\MM'$ have at least $|\pi_{\ell-1}|-1$ points of the type $\pi_i$ so we may choose a fresh $b$ of the same type. The new pair of models found in this manner is clearly atomic-equivalent.

    Phase change: With no changes to the sets of models $\AAA$ and $\BB$, the important pair of models is still clearly present in the next position.

    $\land$-move: S chooses splits $r_1 + r_2 + 1 = r$ and $\BB_1 \cup \BB_2 = \BB$. Now the model $(\MM', s') \in \BB$ is in $\BB_1$ or $\BB_2$ and $\AAA$ remains unchanged. Thus our model pair is present in one of the positions $(r_1,0, \AAA, \BB_1)$ and $(r_2,0, \AAA, \BB_2)$. By choosing such a position, D maintains the pair of models.

    Atomic move: Since the pair of models is atomic-equivalent, D wins after any atomic move.
\end{proof}

The next lemma concerns the atomic phase. We show that if the number of different atomic formulas required to separate the model sets $\AAA$ and $\BB$ is too large, D wins the game.
\begin{lemma}\label{lem:qfreephase}
    In a game $\game(r_0, q_0, \AAA_0, \BB_0)$, let $(r, 0, \AAA, \BB)$ be the first position of the atomic phase and let $\Gamma$ be a minimum size set of atomic formulas such that for every $(\MM, s) \in \AAA$ and $(\MM', s') \in \BB$,  there is $\alpha \in \Gamma$ with $(\MM, s) \vDash \alpha$ and $(\MM',s') \nvDash \alpha$. If $r < 2|\Gamma|-1$, then D has a winning strategy from the position $(r, 0, \AAA, \BB)$.
\end{lemma}
\begin{proof}
    We show that every move S can make either ends the game in a win for D, or maintains the condition $r < 2|\Gamma|-1$ in the next position. Assume this condition holds in position $(r, 0, \AAA, \BB)$.

    Atomic move: S chooses an atomic formula $\alpha$. Since $1 \leq r < 2|\Gamma|-1$, we have $|\Gamma| \geq 2$ so the single atomic formula $\alpha$ does not separate $\AAA$ from $\BB$ and D wins.

    $\land$-move: S chooses splits $r_1 + r_2 + 1 = r$ and $\BB_1 \cup \BB_2 = \BB$. Assume for contradiction that there are sets $\Gamma_1$ and $\Gamma_2$ of atomic formulas such that $\Gamma_i$ separates $\AAA$ from $\BB_i$ and $r_i \geq 2|\Gamma_i|-1$. Now for every pair of models $(\MM, s) \in \AAA$ and $(\MM', s') \in \BB$ we have $(\MM', s') \in \BB_1$ or $(\MM', s') \in \BB_2$ so the set $\Gamma_1 \cup \Gamma_2$ separates $\AAA$ from $\BB$. Recalling that $\Gamma$ is a separating set of minimum size and $r < 2|\Gamma| -1$, we also have $r < 2|\Gamma_1 \cup \Gamma_2|-1 \leq 2(|\Gamma_1| + |\Gamma_2|)-1 \leq r_1 + r_2 + 1 = r$, which is a contradiction. Thus we have $r_1 < 2|\Gamma_1|-1$ or $r_2 < 2|\Gamma_2|-1$. By choosing the correct position D can maintain the required condition.

    $\lor$-move: Identical to the $\land$-move with the roles of $\AAA$ and $\BB$ switched.
\end{proof}

    We are now ready for the main theorem of this section.

    \begin{theorem}\label{thm:lower-bound}
        Let $\MM$ be a model of size $n$. Let $T = \{\pi_1, \dots, \pi_\ell\}$ be the types realized in $\MM$, enumerated in ascending order of numbers of realizing points. Now
        $
            C(\MM) \geq 3|\pi_{l-1}|-3.
        $
    \end{theorem}
    \begin{proof}
    % We only give some intuition on the main idea of the proof here, see \ref{app:lemma9} for the full proof. We let $\Gamma$ be the set of atomic formulas in a separating formula $\varphi$. We consider the quantified variables of $\varphi$ as nodes in a graph, where edges are given by atomic formulas $x = y$ or $x \neq y$ in $\Gamma$. We proceed as in Lemma 7, but instead of finding a fully atomic equivalent pair of models, we find a pair atomic equivalent with respect to each connected component of the variable graph separately. We thus show that $\varphi$ has at least $|\pi_{l-1}|-1$ atomic formulas and together with the quantifiers and connectives this gives us the lower bound. 
    First, we need the following definition. Let $\Gamma$ be a set of atomic $\mathrm{FO}$-formulas. We denote the set of variables occurring in formulas of $\Gamma$ by $V(\Gamma)$. We define the \textbf{variable graph of $\Gamma$} as $G(\Gamma) = (V(\Gamma), E(\Gamma))$, where $(x,y) \in E(\Gamma)$ iff $x = y \in \Gamma$ or $x \neq y \in \Gamma$. We say that $\Delta \subseteq \Gamma$ is a \textbf{connected component of $\Gamma$} if $G(\Delta)$ is a maximal connected subgraph of $G(\Gamma)$.

    For convenience, we denote here $m := |\pi_{\ell-1}|$.

    Consider a formula size game $\game(3m-4, q_0, \AAA_0, \BB_0)$. We show that D has a winning strategy for this game, thus proving the claim by Theorem \ref{thm:gameworks}.

    By Lemma \ref{lem:quantifiers} we see that to have a chance of winning, S must begin the game with at least $m$ quantifiers. We then move on to the first position $(r, 0, \AAA, \BB)$ of the atomic phase, where $r \leq 2m-4$. Let $\Gamma$ be a set of atomic formulas such that for every $(\MM, s) \in \AAA$ and $(\MM', s') \in \BB$, there is $\alpha \in \Gamma$ such that $(\MM, s) \vDash \alpha$ and $(\MM', s') \nvDash \alpha$. If $|\Gamma| \geq m-1$ for every such $\Gamma$, then $r \leq 2m - 4 = 2(m-1)-2 < 2|\Gamma|- 1$ so D has a winning strategy by Lemma \ref{lem:qfreephase}. We now assume for contradiction that there exists such a $\Gamma$ with $|\Gamma| \leq m-2$.

    %Assume first that $\Gamma$ is connected. Since a connected graph with $k$ vertices has at least $k-1$ edges, we see that at most $m-1$ variables occur in the atomic formulas in $\Gamma$. We denote the set of these variables by $V(\Gamma)$.

    Consider the connected components $\Delta$ of $\Gamma$. Since a connected graph with $k$ edges has at most $k+1$ vertices, we see that for every $\Delta$, at most $m-1$ variables occur in the atomic formulas of $\Delta$.
    %$|V(\Delta)| \leq m-1$. 

    We now explain why there is a \emph{single} pair of models $(\MM, s) \in \AAA$ and $(\MM', s') \in \BB$ such that they are atomic equivalent with respect to the variables in $V(\Delta)$ for \emph{every} connected component $\Delta$ of $\Gamma$. We consider the quantifier moves S made in the quantifier phase in the order the moves were made. For every variable $x$ used in a $\exists$-move, we consider $\Delta$ such that $x \in V(\Delta)$. We proceed as in the proof of Lemma \ref{lem:quantifiers}, with respect to only the variables in $V(\Delta)$. That is, if there is a previously quantified variable $y \in V(\Delta)$ such that $s(y) = s(x)$, we choose the opposing model where $s'(x) = s'(y)$. Otherwise, we choose a point with no variables of $V(\Delta)$ attached. Each $\Delta$ uses at most $m-1$ variables so we do not run out of fresh points of any type. The same protocol works for $\forall$-moves as well.

    Note that these choices of models are made based on the connected component $\Delta$ of $x$, completely independently of other components. Since every variable $x$ is in exactly one component $\Delta$, this means that the pair of models we end up with is simultaneously atomic equivalent with regards to each component separately. Thus the pair of models cannot be separated by any atomic formula in $\Gamma$. This is a contradiction with the definition of $\Gamma$ and proves the claim.
    %Now recall that even though we have argued about a single set $\Gamma$ and a pair of models for that set, there are no choices of D in the quantifier phase. The arguments we make merely point out models present in position $P$, rather than leave out any other models. Thus every separating set of literals $\Gamma$ has in position $P$ its own pair of models that cannot be separated with atomic formulas in $\Gamma$. This is a contradiction with the definition of $\Gamma$ and proves the claim.
\end{proof}

    We now consider lower bounds in the setting of $\FO_d$. Recall that an equivalence class of $\equiv_d$ is characterized by a tuple $(m_1, \dots, m_t)$, where $t = 2^{|\tau|}$, $m_i \leq d$, $\sum_{i = 1}^t m_i \leq n$ and if $\sum_{i = 1}^t m_i < n$, then $m_j = d$ for some $j$. 
    Let $\overline{m} = (m_1, \dots, m_t)$ be such a tuple in ascending order of the numbers $m_i$. If $\sum_{i = 1}^t m_i = n$, then $\overline{m}$ corresponds to an isomorphism class and the lower bounds above work as is. Thus we assume that $\sum_{i = 1}^t m_i < n$ and consequently $m_t = d$. 
    By taking a model $\MM$ in the equivalence class $\mathcal{M}_{\overline{m}}$ with a maximal number of points of the type $\pi_t$, we can directly obtain the model $\MM'$ as above and get a lower bound on defining the class $\mathcal{M}_{\overline{m}}$ in full $\FO$. This bound directly extends also to $\FO_d$, as limiting quantifier rank gives no advantage in terms of formula size. 

    %Consider now the model $\MM$, where $|\pi_i| = m_i$ for $i < t$ and the rest of the points realize the type $\pi_t$. Now since $\sum_{i = 1}^t m_i < n$, we have $|\pi_t| > d$ so $\pi_t$ is the type with most realizing points. Additionally, the type $\pi_{t-1}$ has the second most realizing points. The model $\MM'$ obtained by changing the type of one point from $\pi_{t-1}$ to $\pi_t$ is no longer in the equivalence class defined by the tuple $\overline{m}$, since $|\pi_{t-1}| < m_{t-1}$. 

    %The lower bounds we established above for the description complexity of single models in full $\FO$ were based on separating two models $\MM$ and $\MM'$ from each other. Using the two models we defined in the previous paragraph, we obtain a lower bound on defining the equivalence class $\mathcal{M}_{\overline{m}}$ in full $\FO$. Clearly this bound directly extends to $\FO_d$ as well, since limiting quantifier rank gives no advantage in terms of formula size. 
    
    %The lower bounds established above work directly for $\FO_d$ as well. We note that while in $\FO_d$ formulas cannot be in general transformed into prenex normal form without leaving the fragment, this can still be done for formulas with at most $d$ quantifiers in total. Using the models $\MM$ and $\MM'$ we just defined, we have for Lemma \ref{lem:quantifiers} the number $m = |\pi_{\ell-1}| = m_{t-1} \leq d$ so the proof of the lemma only considers formulas with at most $d-1$ quantifiers and prenex normal form causes no issues. 

    \begin{corollary}\label{cor:d-lower-bound}
        Let $\mathcal{M}_{\overline{m}}$ be an equivalence class of $\equiv_d$, where $\overline{m} = (m_1, \dots, m_t)$ is the corresponding tuple with the numbers in ascending order. Now
        $
            C(\mathcal{M}_{\overline{m}}) \geq 3m_{t-1}-3.
        $
    \end{corollary}

    %If $d$ occurs in the tuple $\overline{m}$ at least twice, then we get a lower bound of $3d-3$. 

\section{Expected description complexity}\label{sec:expected-description-complexity}

Using Theorems \ref{thm:upper-bound} and \ref{thm:lower-bound} we can determine asymptotically the expected description complexity of a \emph{random} unary structure. That is, we determine the asymptotic behavior of the quantity
$\mathbb{E}_n[C] := \frac{1}{2^{|\tau|n}}\sum_{\MM} C(\MM).$
Here the sum is taken over all the $\tau$-models $\MM$ of size $n$.

We say that a $\tau$-model $\MM$ is \textbf{balanced},  if for every $\tau$-type $\pi$, we have 
$||\pi|_{\MM} - \frac{n}{2^{|\tau|}}| = o(n).$
A model is balanced, if every type is realized roughly the same number of times, allowing for a sublinear discrepancy. We use the well-known Chernoff bounds to establish that a random model is very likely balanced.

\begin{proposition}[Multiplicative Chernoff bound]\label{prop:chernoff-bound}
    Let $X := \sum_{i = 1}^n X_i$ be a sum of independent $0$-$1$-valued random variables, where $X_i = 1$ with probability $p$ and $X_i = 0$ with probability $1 - p$. Let $\mu := \mathbb{E}[X]$. Now, for every $0 \leq \delta < 1$ we have that
    $
        \Pr[|X - \mu| \geq \delta \mu] \leq 2e^{-\delta^2\mu/3} 
    $
\end{proposition}
\begin{proof}
    See for example Corollary 4.6 in \cite{mitzenmacher2005probability}.
\end{proof}

\begin{lemma}\label{lemma:random-model-is-balanced}
The probability that a random $\tau$-model of size $n$ is balanced is at least $1 - 2^{|\tau| + 1}/n$.
\end{lemma}
\begin{proof}
 We will use Proposition \ref{prop:chernoff-bound}. For every type $\pi$ and $1 \leq i \leq n$ we associate a $0$-$1$-valued random variable $X_{\pi,i}$ such that \(X_{\pi,i} = 1\) with probability \(2^{-|\tau|}\) and \(X_{\pi,i} = 0\) with probability \(1 - 2^{-|\tau|}\). Intuitively this is an indicator random variable for the event ``the $i$th element received the type $\pi$''. Now $X_{\pi} = \sum_{i = 1}^n X_{\pi,i}$ is a random variable that counts the number of times $\pi$ is realized. Clearly $\mathbb{E}[X_\pi] = n/2^{|\tau|}$, which also holds for every type $\pi$. Set $\mu := n/2^{|\tau|}$ and $\delta(n) := \sqrt{\frac{3}{2^{|\tau|}}\frac{\ln(n)}{n}}$. Now
    \[2e^{-\delta(n)^2 \mu / 3} = 2n^{-1}\]
    and
    \[\delta(n)\mu = \frac{\sqrt{3}}{2^{|\tau|}\sqrt{2^{|\tau|}}}\sqrt{\ln(n)n}.\]
    Thus, by Proposition \ref{prop:chernoff-bound}, we know that
    \[
        \Pr\bigg[|X_\pi - \mu| \geq \frac{\sqrt{3}}{2^{|\tau|}\sqrt{2^{|\tau|}}} \sqrt{\ln(n)n}\bigg] \leq 2n^{-1}
    \]
    Applying the union bound, we also see that
    \begin{align*}
    & \Pr\bigg(\exists \pi \ : \ |X_\pi - \mu| 
    \geq \frac{\sqrt{3}}{2^{|\tau|}\sqrt{2^{|\tau|}}} \sqrt{\ln(n)n} \bigg) \\
    & \leq \sum_{\pi} \Pr\bigg[|X_\pi - \mu| \geq \frac{\sqrt{3}}{2^{|\tau|}\sqrt{2^{|\tau|}}} \sqrt{\ln(n)n}\bigg] \\
    & \leq 2^{|\tau| + 1}n^{-1}
    \end{align*}
    Thus, with high probability in a random model $\MM$ of size $n$ we have for every type $\pi$ that
    \[\bigg||\pi|_{\MM} - \frac{n}{2^{|\tau|}}\bigg| \leq \frac{\sqrt{3}}{2^{|\tau|}\sqrt{2^{|\tau|}}}\sqrt{\ln(n)n}.\]
    Hence, with high probability a random model of size $n$ is balanced.
    %A routine calculation using Proposition \ref{prop:chernoff-bound}. See \ref{app:lemma13} for details.
\end{proof}

The previous lemma gives a rough characterization of random $\tau$-models. Using this characterization together with Theorem \ref{thm:lower-bound} we can determine asymptotically the expected description complexity of a random $\tau$-model.

\begin{theorem}
    $\mathbb{E}_n[C] \sim \frac{3n}{2^{|\tau|}}$
\end{theorem}
\begin{proof}
    %We say that a $\tau$-model $M$ is \textbf{balanced}, if for every $\tau$-type $\pi$ we have that $|\pi|_M = n/2^{|\tau|} \pm O(\sqrt{n})$. Using standard concentration bounds, one can show that w.h.p. a random $\tau$-model is good and furthermore that the probability that a random model is not good decreases exponentially fast in $n$.

    To give an upper bound on $\mathbb{E}_n[C]$ we first rewrite it as follows:
    \begin{equation}\label{eq:chopped-up-equation}
        \mathbb{E}_n[C] 
        = \frac{1}{2^{|\tau|n}} \smashoperator[r]{\sum_{\text{$\MM$ balanced}}} C(\MM) + \frac{1}{2^{|\tau|n}} \smashoperator[r]{\sum_{\text{$\MM$ not balanced}}} C(\MM)
    \end{equation}
    Using Corollary \ref{cor:maximum-description-complexity} and Lemma \ref{lemma:random-model-is-balanced} we see that
    \begin{align*}
    & \frac{1}{2^{|\tau|n}} \smashoperator[r]{\sum_{\text{$\MM$ not balanced}}} C(\MM)  \leq \frac{1}{2^{|\tau|n}} \smashoperator[r]{\sum_{\text{$\MM$ not balanced}}} 2n + c_\tau  = \Pr[\text{$\MM$ is not balanced}] \cdot (2n + c_\tau) \\
    & \leq \frac{2^{|\tau| + 1}}{n} \cdot (2n + c_\tau)  = 2^{|\tau| + 2} + \frac{c_\tau2^{|\tau| + 1}}{n}  = \mathcal{O}(1).
    \end{align*}
    Since we are interested in the asymptotic behavior of $\mathbb{E}_n[C]$, the above shows that we can safely concentrate on the first sum in Equation (\ref{eq:chopped-up-equation}). Using Theorems \ref{thm:upper-bound} and \ref{thm:lower-bound} we see that if $\MM$ is balanced, then
    $\frac{3n}{2^{|\tau|}} - o(n) \leq C(\MM) \leq \frac{3n}{2^{|\tau|}} + o(n).$
    Hence\\
    \begin{comment}
        \[
         \Pr[\text{$\MM$ is balanced}] \cdot \bigg(\frac{3n}{2^{|\tau|}} - o(n)\bigg) \leq  \frac{1}{2^{|\tau|n}} \smashoperator[r]{\sum_{\text{$\MM$ balanced}}} C(\MM) 
         \leq \Pr[\text{$\MM$ is balanced}] \cdot \bigg(\frac{3n}{2^{|\tau|}} + o(n)\bigg).
        \]
        \end{comment}
    \begin{align*}
        & \Pr[\text{$\MM$ is balanced}] \cdot \bigg(\frac{3n}{2^{|\tau|}} - o(n)\bigg) \leq  \frac{1}{2^{|\tau|n}} \smashoperator[r]{\sum_{\text{$\MM$ balanced}}} C(\MM) \\ 
        & \leq \Pr[\text{$\MM$ is balanced}] \cdot \bigg(\frac{3n}{2^{|\tau|}} + o(n)\bigg)
    \end{align*} 

    Since $\Pr[\text{$\MM$ is balanced}]$ goes to one as $n \to \infty$, we see that 
    \[\dfrac{1}{2^{|\tau|n}} \smashoperator[r]{\sum_{\text{$\MM$ balanced}}} C(\MM) \sim \frac{3n}{2^{|\tau|}},\]
    which is what we wanted to show.
\end{proof}

\section{Entropy and description complexity}\label{sec:entropy}

In this section we establish results that illustrate how entropy and description complexity relate to each other. As one can already imagine after seeing our results on description complexity, there can be models with very close entropies and quite different description complexities. We can nevertheless use our results to \emph{exclude} many a priori possible combinations of description complexity and entropy.
%obtain clear bounds on this behaviour. 
For notational simplicity, we adopt the notation $t := 2^{|\tau|}$ in this section.

We begin by showing that the Boltzmann and Shannon entropies of a single model are essentially the same up to normalization. This underlines the fact that both entropies measure the same thing: the randomness of a model. 

\begin{theorem}\label{thm:entropy-connection}
    Let $\MM$ be a $\tau$-model of size $n$. Now 
    \[H_S(\MM) - \frac{1}{n}H_B(\MM) < \frac{(t-1)\log(\sqrt{2 \pi n})}{n} - \frac{\log(e)}{12n^2} + \frac{t\log(e)}{12n^2+n}.\] 
\end{theorem} 
\begin{proof}
    %A calculation via the quantitative version of Stirling's approximation given in \cite{stirlingappr}. See \ref{app:thm15} for the details.
    Using the quantitative version of Stirling's approximation given in 
    \cite{stirlingappr}, we obtain
    \begin{align*}
        H_B(\MM) = &\log\binom{n}{n_1 \dots n_t}  
        =\  \log\frac{n!}{n_1!\dots n_t!} = \log(n!) - \sum\limits_{i = 1}^t \log(n_i!) \\
        <\ &\log\bigg(\sqrt{2\pi n}\bigg(\frac{n}{e}\bigg)^n e^\frac{1}{12n}\bigg)-\sum\limits_{i = 1}^t \log\bigg(\sqrt{2\pi n_i} \bigg(\frac{n_i}{e}\bigg)^{n_i} e^\frac{1}{12n+1}\bigg) \\
        =\ &\log(\sqrt{2\pi n}) + n\log(n) - n\log(e) + \frac{\log(e)}{12n} \\ 
        &- \sum\limits_{i = 1}^t \bigg(\log(\sqrt{2\pi n_i}) + n_i \log(n_i) - n_i\log(e) + \frac{\log(e)}{12n+1}\bigg) \\
        \leq\ &n\log(n) - \sum\limits_{i = 1}^t n_i \log(n_i) - (t-1)\log(\sqrt{2 \pi n}) + \frac{\log(e)}{12n} - \frac{t\log(e)}{12n+1}.
    \end{align*}
    Note that the term $n \log(e)$ is cancelled out above because $n_1 + \dots + n_t = n$. Using this same fact we also easily see that
    \begin{align*}
        H_S(\MM) &= \sum\limits_{i = 1}^t -\frac{n_i}{n}\log\frac{n_i}{n} = \sum\limits_{i = 1}^t \frac{n_i}{n}\log(n) - \sum\limits_{i = 1}^t \frac{n_i}{n}\log(n_i) \\
        &= \log(n) - \sum\limits_{i = 1}^t \frac{n_i}{n}\log(n_i).
    \end{align*}
    Finally, by dividing $H_B(\MM)$ with $n$ we obtain
    \begin{align*}
        H_S(\MM) - \frac{1}{n}H_B(\MM) < \frac{(t-1)\log(\sqrt{2 \pi n})}{n} - \frac{\log(e)}{12n^2} + \frac{t\log(e)}{12n^2+n}.
    \end{align*}
\end{proof}

The above quantitative result readily implies that the Boltzmann and Shannon entropies of a single model are asymptotically the same up to normalization. 
A connection that bears a similarity to the one pointed out here has also
been noted briefly in \cite{Kolmogorov1993}.
%\textcolor{red}{We should state in more detail how similar and 
%for what structures.}

\begin{corollary}
    Let $(\MM_n)_{n \in \mathbb{Z}_+}$ be a sequence of $\tau$-models where each $\MM_n$ has size $n$. Now
    $H_S(\MM_n) \sim \frac{1}{n}H_B(\MM_n) \text{ as } n \to \infty.$
    
\end{corollary}

The above results show that for the connections to description complexity, we could use either of the two notions of entropy. We opt for Shannon entropy here. 

We will next use results from Sections \ref{sec:upperbounds} and \ref{sec:lower-bounds} to prove two theorems that give bounds on description complexity in terms of Shannon entropy. Recall from Section \ref{sec:upperbounds} the constant $c_\tau := 15|\tau|2^{|\tau|}$.

The first of our two theorems gives global upper and lower bounds on description complexity based on the same edge case distributions.

\begin{theorem}\label{thm:entropy-lower-upper}
    Let $p \in [0,\frac{1}{t}[$. If $H_S(\MM) > ((t-1)p-1)\log(1-(t-1)p) - (t-1)p\log(p),$ then $3np-3 < C(\MM)< 3n(1-(t-1)p) + c_\tau.$
\end{theorem} 
\begin{proof}
    Let $f(p) := ((t-1)p-1)\log(1-(t-1)p) - (t-1)p\log(p)$. The function $f(p)$ gives the entropy of a $\tau$-model $\MM'$ corresponding to the tuple $(np, \dots, np, n(1-(t-1)p))$, where $n(1-(t-1)p) > np$ for the given values of $p$.
    %, where all types, except one, have the same number of realizing points, that is, $np$. The final type has $n(1-(t-1)p)$ realizing points, which is more than the other types (for the given values of $p$). 
    Since all types but the largest are evenly distributed, any model, where the largest type has at least $n(1-(t-1)p)$ realizing points has entropy at most $H_S(\MM') = f(p)$. Therefore if $H_S(\MM) > f(p)$, then the largest type of $\MM$ has less than $n(1-(t-1)p)$ realizing points. By Theorem \ref{thm:upper-bound}, we obtain $C(\MM) < 3n(1-(t-1)p) + c_\tau$. On the other hand, since the largest type of $\MM$ has less realizing points than in $\MM'$, those points realize some other type. Therefore the second largest type of $\MM$ has more than $np$ realizing points. By Theorem \ref{thm:lower-bound}, we obtain $C(\MM) > 3np-3$. 
\end{proof}

\begin{comment}
\begin{theorem}
    Let $p \in [0,\frac{1}{t}]$. If 
    \[H(\MM) > ((t-1)p-1)\log(1-(t-1)p) - (t-1)p\log(p),\] then \[C(\MM) < 3n(1-(t-1)p) + c_\tau.\]
\end{theorem}
\begin{proof}
    Let $f(p) := ((t-1)p-1)\log(1-(t-1)p) - (t-1)p\log(p)$. The function $f(p)$ gives the entropy of a $\tau$-model, where all types, except the largest one, are of the same size, that is, $np$, while the largest one is of size $n(1-(t-1)p)$. If $H(\MM) > f(p)$, then the largest type of $\MM$ must be smaller than $n(1-(t-1)p)$. Therefore by Theorem \ref{thm:upper-bound}, we obtain $C(\MM) < 3n(1-(t-1)p) + c_\tau$.
\end{proof}
\end{comment}

The next theorem uses low entropy models with only two realized types to show a better upper bound on description complexity for low entropy models than the above global one.

\begin{theorem}\label{thm:entropy-upper}
    Let $p \in [0,\frac{1}{2}]$. 
    If $H_S(\MM) < (p-1)\log(1-p)-p\log(p)$, then $C(\MM) < 6np + c_\tau.$
\end{theorem}
\begin{proof}
    Let $h(p) := (p-1)\log(1-p)-p\log(p).$ The function $h(p)$ gives the entropy of a $\tau$-model $\MM$ corresponding to the tuple $(0, \dots, 0, np, n(1-p))$.
    %, where the largest type has size $n(1-p)$, the second largest has size $np$ and the rest of the types are not realized. 
    If $H_S(\MM) < h(p)$, then the second largest type of $\MM$ must be smaller than $np$. 
    %Additionally, since $p \leq \frac{1}{3}$, the largest type is at least twice as large as the second largest. Therefore, by Theorem \ref{thm:upper-bound}, we may use the upper bound based on the second largest type. 
    Thus, by Theorem \ref{thm:upper-bound}, $C(\MM) < 6np + c_\tau$.
\end{proof}

\begin{figure}[htb]
\centering
\captionsetup[subfigure]{justification=centering}
\begin{subfigure}{.5\textwidth}
    \centering
    \resizebox{0.9\textwidth}{!}{
\begin{tikzpicture}%[every node/.style={scale=0.4}]
    \draw[thick, ->] (0,0) --  (2.1*3,0);
    \draw[thick, ->] (0,0) -- (0,2.3*1*3);

    \node at (2.3*3, 0) {$H_S(\MM)$};
    \node at (0, 2.4*1*3) {$C(\MM)$};

    \draw[thick] (1*3,.03*3) -- (1*3,-.03*3) node[below] {1};
    \draw[thick] (2*3,.03*3) -- (2*3,-.03*3) node[below] {2};

    \draw[thick] (.03*3,0.288) -- (-.03*3,0.288) node[left] {$c_\tau$};
    %\draw[thick] (.03*3,3/4*1*3+0.288) -- (-.03*3,3/4*1*3+0.288) node[left] {$3n/4+c_\tau$};
    \draw[thick] (.03*3,3/4*1*3) -- (-.03*3,3/4*1*3) node[left] {$3n/4$};
    %\draw[thick] (.03*3,3/2*1*3+0.288) -- (-.03*3,3/2*1*3+0.288) node[left] {$3n/2+c_\tau$};
    \draw[thick] (.03*3,3/2*1*3) -- (-.03*3,3/2*1*3) node[left] {$3n/2$};
    \draw[thick] (.03*3,2*1*3+0.288) -- (-.03*3,2*1*3+0.288) node[left] {$2n+c_\tau$};

    %\draw[fill] (0,4/80*2) circle[radius=1.5pt];
    %\draw[fill] (0,0) circle[radius=1.5pt];

    %\draw[fill] (0.918*3,2*1*3+0.288) circle[radius=1.5pt];
    %\draw[fill] (0.918*3,1*1*3) circle[radius=1.5pt]; 
    %\draw[thick, dash pattern=on 1pt off 1pt] (0.918*3,2*1*3+0.288) -- (0.918*3,1*1*3);
    
    %\draw[fill] (1*3, 3/2*1*3) circle[radius=1.5pt];
    %\draw[fill] (1*3, 3/2*1*3+0.288) circle[radius=1.5pt];
    %\draw[thick, dash pattern=on 1pt off 1pt] (1*3, 3/2*1*3) -- (1*3, 3/2*1*3+0.288);

    %\draw[fill] (1.58*3, 1*1*3) circle[radius=1.5pt];
    %\draw[fill] (1.58*3, 1*1*3+0.288) circle[radius=1.5pt];
    %\draw[thick, dash pattern=on 1pt off 1pt] (1.58*3, 1*1*3) -- (1.58*3, 1*1*3+0.288);
    %\draw[fill] (1.52*3, 3*0.288*1*3) circle[radius=1.5pt];
    %\draw[fill] (1.16*3, 3*0.64*1*3) circle[radius=1.5pt];
    %\draw[fill] (1.92*3, 3*0.288*1*3) circle[radius=1.5pt];

    %\draw[fill] (2*3, 3/4*1*3) circle[radius=1.5pt];
    %\draw[fill] (2*3, 3/4*1*3+0.288) circle[radius=1.5pt];
    %\draw[thick, dash pattern=on 1pt off 1pt] (2*3, 3/4*1*3) -- (2*3, 3/4*1*3+0.288);

    \draw[very thick, variable=\y, domain={(1/3-1)*log2(1-1/3)-1/3*log2(1/3)}:{((4-1)*1/9-1)*log2(1-(4-1)*1/9) - (4-1)*1/9*log2(1/9)}] plot (\y*3,2*1*3+0.288);

    \draw[very thick, variable=\y, domain=1/9:1/4] plot ({(((4-1)*\y-1)*log2(1-(4-1)*\y) - (4-1)*\y*log2(\y))*3},{3*(1-(4-1)*\y)*1*3+0.288});
    
    \draw[very thick, variable=\y, domain=0.001:1/3] plot ({((\y-1)*log2(1-\y)-\y*log2(\y))*3},6*\y*1*3+0.288);

    \draw[very thick,variable=\y, domain=0.001:1/18] plot ({(((4-1)*\y-1)*log2(1-(4-1)*\y) - (4-1)*\y*log2(\y))*3},3*\y*1*3);
    \draw[very thick, variable=\y, domain=1/18:1/4] plot ({(((4-1)*\y-1)*log2(1-(4-1)*\y) - (4-1)*\y*log2(\y))*3},3*\y*1*3);
    
    %\draw[variable=\y, domain=1/4:0.5] plot ({-2*\y*log2(\y)-(1-2*\y)*log2((1-2*\y)/(4-2))},6*\y*1*3);
\end{tikzpicture}
}
    \caption{}
    %An area that encapsulates all combinations of Shannon entropy and description complexity for the values $|\tau| = 2$ and $n = 1000$. The solid lower bound curve comes from Theorem \ref{thm:entropy-lower-upper}. The three parts of the dashed upper bound curve are from Theorem \ref{thm:entropy-upper}, Corollary \ref{cor:maximum-description-complexity} and Theorem \ref{thm:entropy-lower-upper}, respectively.}

    \label{fig:entropy}
\end{subfigure}%
\begin{subfigure}{.5\textwidth}
\centering
    \resizebox{0.9\textwidth}{!}{
\begin{tikzpicture}%[every node/.style={scale=0.4}]
    \draw[thick, ->] (0,0) --  (2.1*3,0);
    \draw[thick, ->] (0,0) -- (0,2.1*1*3);

    \node at (2.3*3, 0) {$H_B^d(\MM)$};
    \node at (0, 2.2*1*3) {$C_d(\MM)$};

    \draw[thick] (.03*3,2*1*3) -- (-.03*3,2*1*3) node[left] {$60$};
    \draw[thick] (.03*3,1*1*3) -- (-.03*3,1*1*3) node[left] {$30$};

    \draw[thick] (1*3,.03*3) -- (1*3,-.03*3) node[below] {100};
    \draw[thick] (2*3,.03*3) -- (2*3,-.03*3) node[below] {200};

    \draw[thick, fill=white] (0, 2/10*1*3) circle[radius=1.5pt];
    \draw[fill] (0, 0) circle[radius=1.5pt];

    \draw[thick] (19.9/100*3, 1/10*1*3) -- (36.6/100*3, 1/10*1*3);
    \draw[thick] (36.6/100*3, 2/10*1*3) -- (51.1/100*3, 2/10*1*3);
    \draw[thick] (65.0/100*3, 3/10*1*3) -- (51.1/100*3, 3/10*1*3);
    \draw[thick] (65.0/100*3, 4/10*1*3) -- (77.3/100*3, 4/10*1*3);
    \draw[thick] (88.8/100*3, 5/10*1*3) -- (77.3/100*3, 5/10*1*3);
    \draw[thick] (88.8/100*3, 6/10*1*3) -- (99.3/100*3, 6/10*1*3);
    \draw[thick] (109.2/100*3, 7/10*1*3) -- (99.3/100*3, 7/10*1*3);
    \draw[thick] (109.2/100*3, 8/10*1*3) -- (118.4/100*3, 8/10*1*3);
    \draw[thick] (200/100*3, 9/10*1*3) -- (118.4/100*3, 9/10*1*3);

    \draw[thick] (6.6/100*3, 2/10*1*3) -- (0/100*3, 2/10*1*3);
    \draw[thick] (6.6/100*3, 4/10*1*3) -- (12.3/100*3, 4/10*1*3);
    \draw[thick] (17.3/100*3, 6/10*1*3) -- (12.3/100*3, 6/10*1*3);
    \draw[thick] (17.3/100*3, 8/10*1*3) -- (21.9/100*3, 8/10*1*3);
    \draw[thick] (26.2/100*3, 10/10*1*3) -- (21.9/100*3, 10/10*1*3);
    \draw[thick] (26.2/100*3, 12/10*1*3) -- (30.2/100*3, 12/10*1*3);
    \draw[thick] (33.9/100*3, 14/10*1*3) -- (30.2/100*3, 14/10*1*3);
    \draw[thick] (33.9/100*3, 16/10*1*3) -- (37.4/100*3, 16/10*1*3);
    \draw[thick] (40.8/100*3, 18/10*1*3) -- (37.4/100*3, 18/10*1*3);
    \draw[thick] (40.8/100*3, 19/10*1*3) -- (200/100*3, 19/10*1*3);

    \draw[thick, fill=white] (200/100*3, 19/10*1*3) circle[radius=1.5pt];
    \draw[fill] (200/100*3, 9/10*1*3) circle[radius=1.5pt];
    
\end{tikzpicture}
}
\caption{}
%Bounds on description complexity in terms of entropy for values $|\tau| = 2$, $n = 100$ and $d= 10$ with no constants. The lower plot comes from Theorem \ref{thm:d-lower}, while the upper one is from Theorem \ref{thm:d-upper}. The filled circles mark the points where the two plots overlap.}
\label{fig:d-entropy}
\end{subfigure} 
\caption{Figure \ref{fig:entropy} on the left shows an area that encapsulates all combinations of Shannon entropy and $\FO$-description complexity for the values $|\tau| = 2$ and $n = 1000$. Figure \ref{fig:d-entropy} on the right concerns the case of $\FO_d$ and shows bounds on description complexity in terms of Boltzmann entropy for values $|\tau| = 2$, $n = 100$ and $d= 10$ with the constants $-3$ and $c_\tau$ omitted.}

\end{figure}

Figure \ref{fig:entropy} incorporates both of the above theorems as well as Corollary \ref{cor:maximum-description-complexity} to show an area, where all possible combinations of Shannon entropy and description complexity must fall. First, comparing the left side of the plot to the right, we can see that models with very high entropy have significantly higher description complexity than models with very low entropy. 

We can also see from Figure \ref{fig:entropy} that the gap between our upper bounds and lower bounds is only constant at both extremes of entropy. For models with middling entropy, the gap is at its largest. This is because middling values of entropy can be realized by models with very different distributions of types, leading to different description complexity. 
%In addition, starting from the left, we can see that entropy and description complexity both increase together at first. Around the midpoint of entropy, and onwards from there, we see this connection ultimately breaking down. 

 We conjecture that the upper bound given by Theorem \ref{thm:upper-bound} is in reality tight up to the constant $c_\tau$. Now, recall that for any single model, our upper and lower bounds have a worst case gap of a factor of 2. 
 %Therefore, if our conjecture was true, the plot of the lower bound curve in Figure \ref{fig:entropy} would rise more steeply before eventually ending up at the same end point. 
 Therefore, assuming that our conjecture is true, the lower bound would only rise to at most double its current height. In other words, the general picture illustrated by Figure \ref{fig:entropy} would not be significantly different under our conjecture.

\begin{comment}
\begin{theorem}
    There is a constant $h \approx 0.92$ and strictly increasing functions $f, g : [0,h] \to \mathbb{R}_{\geq 0}$ such that
    \[
        f(H_S(\MM)) < C(\MM) < g(H_S(\MM)),
    \]
    for all $\MM$ with $H_S(\MM) \in [0,h]$.
\end{theorem}
\end{comment}

We proceed to show that similar relationships between description complexity and entropy hold also in the case of limited quantifier rank. 
As the classes of $\equiv_d$ contain multiple different isomorphism types of models, it is not clear how to define Shannon entropy. Boltzmann entropy, however, is still straightforward so we use Boltzmann entropy here. 
We formulate similar theorems to those above for full $\FO$. 
%Just like for full $\FO$, the lower bound is based on models with an even distribution of all types but one.

\begin{theorem}\label{thm:d-lower}
    Let $h \in \{1,..., d-1\}$. 
    If $H_B^d(\MM) > \log\binom{n}{h \dots h \  n-(t-1)h}$,  then  $C_d(\MM) > 3h-3$.
\end{theorem}
\begin{proof}
    Let $f(n,h) = \log\binom{n}{h \dots h \  n-(t-1)h}$. The function $f(n,h)$ gives the Boltzmann entropy of the class of models $\mathcal{M}_{\overline{m}}$, where $\overline{m} = (h, \dots, h, d)$. Any class of models obtained from this one by lowering any of the numbers in the tuple is clearly smaller than $\mathcal{M}_{\overline{m}}$ and thus has lower Boltzmann entropy. Thus, for any larger class of models the second largest number in its tuple must be greater than $h$. By Corollary \ref{cor:d-lower-bound}, we obtain $C_d(\MM) > 3h -3$.
\end{proof}

%The upper bound is again based on models with only two realized types.

\begin{theorem}\label{thm:d-upper}
    Let $h \in \{1, \dots, d-1\}$. 
    If  $H_B^d(\MM) < \log\binom{n}{h}$, then  $C_d(\MM) < 6h+c_\tau$.
\end{theorem} 
\begin{proof}
    The function $g(n,h) = \log\binom{n}{h}$ gives the Boltzmann entropy of a class $\mathcal{M}_{\overline{m}}$ of models, where $\overline{m} = (0, \dots, 0, h, d)$. Now every class of models, where the second largest number in the tuple is at least $h$, is larger than or equal to $\mathcal{M}_{\overline{m}}$. Thus if $H_B^d(\MM) < g(n,h)$, then the class of $\MM$ is smaller and the second largest number in its tuple is smaller than $h$. By Theorem \ref{thm:d-upper-bound} we obtain $C_d(\MM) < 6h+c_\tau$.
\end{proof}

We again have a plot in Figure \ref{fig:d-entropy}, where the possible combinations of entropy and description complexity lie between the two chopped lines. This time, we plotted from the above theorems $3h$ for the lower bound and $6h$ for the upper bound, omitting the constants $-3$ and $c_\tau$. For these low values of $n$ and $d$, the constants would have warped the picture in a significant way. With high enough $n$ and $d$, the constants are clearly negligible, but for such values, the Boltzmann entropy quickly becomes impractical to calculate as the model class sizes explode. We provide a plot of the leading terms for the values $n = 100$ and $d = 10$ without the constants to illustrate the trends one would see for higher values of $n$ and $d$. 

We see that the first observation we made for full $\FO$ still holds. The models with  very high entropy have significantly higher description complexity than those with very low entropy. 
Concerning the gap between the upper and lower bounds, it is again constant at the extremes. The largest gap can now be found significantly before the halfway point of entropy, unlike for full $\FO$. This is because the limit $d$ of quantifier rank quite quickly cuts short the growth of the upper bound while the lower bound grows slower. %This will of course vary according to the relationship of $d$ and $n$.

%Starting from entropy $0$, we can again observe a trend of both measures increasing before the upper bound reaches the maximum and the situation becomes more unclear.

\begin{comment}
\begin{theorem}
    Let $p \in ]\frac{1}{t},\frac{1}{2}]$. If
    \[
        H(\MM) > -2p\log(p)-(1-2p)\log(\frac{1-2p}{t-2}),
    \]
    then
    \[
        C(\MM) < 6np + c_\tau.
    \]
\end{theorem}
\begin{proof}
    Let $g(p) := -2p\log(p)-(1-2p)\log(\frac{1-2p}{t-2})$. The function $g(p)$ gives the entropy of a $\tau$-model, where the two largest types are of the size $np$ and the rest of the types are of equal size with each other, usually smaller than $np$. If $H(\MM) > g(p)$, then it must be the case that the second largest type of $\MM$ is smaller than $np$. Therefore, by Theorem \ref{thm:upper-bound}, we see that $C(\MM) < 6np + c_\tau$.
\end{proof}
\end{comment}

\section{Conclusion}\label{sec:conclusion}

We have studied the description complexity of unary models, obtaining bounds for $\FO$ and $\FO_d$. We have found the asymptotic description complexity of a random unary structure and studied the relation between Shannon entropy and description complexity---also observing a connection between Boltzmann and Shannon entropy. Links to entropy can be useful as computing entropy is \emph{significantly easier than determining description complexity.}

An obvious future goal would be to close the gaps between the upper and lower bounds. Generalizing to full relational vocabularies is also interesting, although this seems to require highly involved arguments. The part on entropy would there relate to Boltzmann entropy, as there is no obvious unique definition for Shannon entropy in the $k$-ary scenario.

\bibliographystyle{plain}
\bibliography{uusi_arxiv}

\end{document}